\documentclass[a4paper,11pt]{article}         

\usepackage{graphicx}
\usepackage{amsmath}  
\usepackage{amsthm}
\usepackage{amsfonts}
\usepackage{lmodern}
\usepackage{graphicx}
\usepackage{array} % Extension of the tabular environment.
\usepackage{enumerate} % Consider use of package `paralist'.
\usepackage[ansinew]{inputenc} %Input encoding.
\usepackage[T1]{fontenc} %Output encoding.  
\usepackage{url}
\usepackage[usenames]{xcolor}
\usepackage{lscape}
\usepackage{framed}\colorlet{shadecolor}{black!25!white} % Provisional for frames
\usepackage{fancyhdr}
\fancyheadoffset{30pt} %Fa la capçalera més llarga que l'amplada del text.
\fancyfootoffset{30pt} %Fa el peu més llarg que l'amplada del text.

\newcommand\Comment[1]{{\color{blue}#1}} 
\newcommand\unio{\mathop\cup\limits}
\newcommand{\Reals}{\mathbb{R}}  
\newcommand{\Naturals}{\mathbb{N}}  
\newcommand{\Rationals}{\mathbb{Q}}

\newcommand{\E}{\operatorname{E}}
\newcommand{\Cov}{\operatorname{Cov}}
\newcommand{\Var}{\operatorname{Var}}
\newcommand{\ind}{\mathbf{1}}
\newcommand{\condprob}[2]{\raise2pt
                   \hbox{%
                   \mathsurround=0pt$#1$}
                    \ \raise-1pt\hbox{\scalebox{1.2}[1.5]{/}}\ 
                   \raise-2pt
                   \hbox{%
                   \mathsurround=0pt$#2$}
                    }

\newtheorem{theorem}{Theorem}[section]
\newtheorem{dfn}[theorem]{Definition}
\newtheorem{lemma}[theorem]{Lemma}
\newtheorem{proposition}[theorem]{Proposition}
\newtheorem{remark}[theorem]{Remark}

\newtheorem{examples}[theorem]{Examples}
\newtheorem{corollary}[theorem]{Corollary}

\newlength\savedwidth

\begin{document}  

\title{No-Free-Lunch Theorems in the Continuum
%\thanks{Grants or other notes}
}
\author{Aureli Alabert \\ %\at
           Department of Mathematics \\
           Universitat Aut\`onoma de Barcelona \\  
           08193 Bellaterra, Catalonia \\  
           \url{Aureli.Alabert@uab.cat}  
           \and
           Alessandro Berti \\ %\at
           Dipartimento di Matematica \\
           Universit\`a di Padova \\  
           35121 Padova, Italy \\  
           \url{berti@math.unipd.it} 
           \and
           Ricard Caballero \\ %\at
           Department of Mathematics \\
           Universitat Aut\`onoma de Barcelona \\  
           08193 Bellaterra, Catalonia \\  
           \url{rcaballero@mat.uab.cat}  
           \and
           Marco Ferrante \\ %\at
           Dipartimento di Matematica \\
           Universit\`a di Padova \\  
           35121 Padova, Italy \\  
           \url{ferrante@math.unipd.it}
}

\date{\today}
\maketitle
\thispagestyle{empty}
\begin{abstract}  
 No-Free-Lunch Theorems state, roughly speaking, that the performance of all search algorithms is the
 same when averaged over all possible objective functions. This fact was precisely formulated for the first
 time in a now famous paper by Wolpert and Macready, and then subsequently refined and extended by several
 authors, always in the context of a set of functions with discrete domain and codomain. Recently, Auger  
 and Teytaud have shown that for continuum domains there is typically no No-Free-Lunch theorems. In this paper
 we provide another approach, which is simpler, requires less assumptions, relates the discrete and 
 continuum cases,
 and that we believe that clarifies 
 the role of the cardinality and structure of the domain. 
\end{abstract}

\section{Introduction}\label{sec:Intro}
In \cite{Wolpert97nofree}, Wolpert and Macready formulated rigorously a principle which was already 
intuitively known to the operations research practitioners: All search or optimization algorithms perform equally 
well when their performance is averaged against all possible objective functions. This principle has been
known since then as the \emph{No-Free-Lunch Theorem} (NFL for short). 

The precise formulation of the Wolpert-Macready NFL Theorem will be stated in Section \ref{sec:Prelim} (Theorem \ref{thm:WMNFL}), 
but the basic assumptions are that we are dealing with the set of all functions $f\colon{\cal X}\rightarrow {\cal Y}$ between two finite sets
${\cal X}$ and ${\cal Y}$, and that the ``averaging'' is uniform over all these functions. The measure of  
performance can be any function of the images $f(x_1),\dots, f(x_m)$ of the points $x_1, \dots, x_m$ 
sampled by the algorithm.

In \cite{Schumacher01theno}, Schumacher, Vose and Whitley extended the result to some subsets of all functions
(those called ``closed under permutation''), whereas Igel and Toussaint \cite{1079.90111} stated it 
for some non-uniform measures.
The language of probability theory allows to formulate these
statements  
in a unified and easier way and, as we will see, it is absolutely necessary to switch from the discrete
(finite) setting to the continuum.
In \cite{MR2581081}, Auger and Teytaud considered for the first time this case, 
and their result is essentially negative: No NFL theorems exist in the continuum.

%\Comment{((To review. Important.))} 
Our goal in this paper is to improve and clarify the results of Auger and Teytaud, particularly 
Theorem 4.1, \cite{MR2581081}.
First of all, we show situations where NFL theorem do exist. This apparent paradox is resolved
by noticing that the hypotheses imposed in  \cite{MR2581081} invalidate our examples. In fact, 
the authors  seem to specifically look for conditions under which no NFL theorem can hold true. 
The theorem is indeed correct, although there is a gap in the proof, as explained in Section \ref{sec:Main}.
We must also point out, however, that their paper contains much more 
material of interest on this and other matters. 

The point of view adopted here is different: We establish a simple and natural definition
of  the NFL property and look for the necessary conditions implied by this definition.
In this sense, our main result is Theorem \ref{thm:NFLisConst}. The conclusion we reach is that 
there are no No-Free-Lunch theorems for
functions whose domain is the real number system, except for a few extreme cases.
Admittedly, the theorem contains an additional mild technical condition (the existence of
second-order moments) that we have not yet been able to remove.

%\Comment{((To review))} 
The relevance of this theoretical discussion for the field of global \emph{black-box} optimisation 
comes from the so-called \emph{probabilistic models}:  In many practical optimisation problems there 
is little information about the objective function, with no access to derivatives or to any explicit formula;
we are only allowed to ask the function for its value at a point of our choice and, after observing the 
value returned, we may decide on the next point to
sample the function; and so on. Moreover, function evaluations can be expensive, 
%so that it is feasible to make only a small number
%of them. 
and we are constrained to make only a small number of them.
In these cases, it may be useful to think that the function has been
drawn at random from some set of functions, according to some probability law (perhaps with some
unknown parameters) that one specifies using prior information. Technically, we are then 
in the presence of a \emph{stochastic process}, from
which our function is a particular path. 
Different algorithms will choose different points for the successive evaluations, and some may perform
better than others by exploiting better the model, \emph{unless} there is a No-Free-Lunch
theorem for that model. If this is the case, all algorithms perform the same and, in particular, pure blind search
is as good as any other proposal. In the present paper we will see that the presence of the No-Free-Lunch 
property reduces to a few probabilistic models, which are not
really important in practice.

%\Comment{((To review))} 
The paper is organised as follows: In Section \ref{sec:Prelim} we state 
the definitions and preliminaries both from algorithmics and from probability theory that are strictly
needed in the rest of the paper. In Section \ref{sec:Main} we state the main results: We show that 
No-Free-Lunch cases do exist in the continuum; we impose then a hypothesis of measurability of the 
stochastic process involved, and we see that NFL can only appear if we are essentially in a discrete 
setting (Theorem \ref{thm:MesInd}),
or the model consists of a trivial constant process (Theorem \ref{thm:NFLisConst}). In Section
\ref{sec:Conc} we justify the investigation of the existence (or not) of NFL properties in the continuum
and propose some open questions.

\section{Preliminaries}\label{sec:Prelim}
We follow approximately the notations of \cite{1079.90111} and \cite{MR2581081}, with some convenient 
modifications.

\subsection{Algorithmic concepts}

Let ${\cal X}$ and ${\cal Y}$ be any two sets. The set of all functions $f\colon{\cal X}\rightarrow{\cal Y}$ can be
identified with the Cartesian product ${\cal Y}^{\cal X}$. Denote
\begin{equation}
{\cal E}_0:=\{\emptyset\},\ 
{\cal E}_1:={\cal X}\times {\cal Y},\ 
\dots,\ 
{\cal E}_m={\cal X}^m\times {\cal Y}^m
\end{equation}
and ${\cal E}:=\unio_{m\ge 0}{\cal E}_m$.

A \emph{(random) algorithm} $A$ is a mapping $A\colon {\cal Y}^{\cal X}\times{\cal E}\times\Theta\rightarrow {\cal E}$, 
where $(\Theta,{\cal G}, Q)$ is a probability space and, 
if $e=((x_1,y_1),\dots,(x_m,y_m))\in{\cal E}_m$, then 
$A(f, e,\theta)\in{\cal E}_{m+1}$ and 
\begin{equation}
A(f, e,\theta)=((x_1,y_1),\dots,(x_m,y_m),(x_{m+1},y_{m+1}))
\end{equation}
with $y_{m+1}=f(x_{m+1})$. Therefore, we can think of  $A(f, e):=(X_1,Y_1),\dots,(X_m,Y_m))$ 
as a random vector $\Theta\rightarrow {\cal E}$.
This definition formalises the fact that the algorithm chooses the next point based on the 
previous points and an (optional) random mechanism represented by the probability
space $(\Theta,{\cal G}, Q)$.  
One may assume that $f$ is never evaluated more than once at the same point.

A \emph{measure of performance} of the algorithm is any function $C$ of the values 
obtained by evaluating $f$ during the algorithm. Formally,
\begin{equation}
C\colon\unio_{m\ge 1} {\cal Y}^m\rightarrow \Reals 
\ .
\end{equation}
A typical measure of performance for optimization problems is the function
$C(y_1,\dots,y_m)=\min\{y_1,\dots,y_m\}$, the best observed value after $m$ evaluations. 
(Notice that the measures of performance we are talking
about are not related to algorithmic complexity, e.g. to the number of evaluations needed 
to reach the end of a procedure.)

To state the basic Wolpert-Macready Theorem, rephrased in our probability-theoretic language,
consider another probability space $(\Omega, {\cal F}, P)$, and a random variable
$f\colon\Omega\rightarrow {\cal Y}^{\cal X}$. Now $f$ is random and $f(\omega)$, for each $\omega$,
is a specific function ${\cal X}\rightarrow{\cal Y}$.
Denote by $A^m\colon{\cal Y}^{\cal X}\times\Theta\rightarrow{\cal E}_m$ the successive application, $m$ times, 
of algorithm $A$ to the initial empty 
sequence $\emptyset$, and by $A^m_Y\colon{\cal Y}^{\cal X}\times\Theta\rightarrow{\cal Y}^m$ its
second component. Finally, let us abbreviate 
$y:=(y_1,\dots,y_m)\in {\cal Y}^m$.

 The conditional probability 
 $Q\big\{\condprob{A_Y^m (h, \theta) = y }{ f(\omega)=h}\big\}$ is the probability 
 that the algorithm 
  $A$ produce the particular sequence of function values 
  $y:=(y_1,\dots,y_m)$ when applied to the function $h$. This probability is either
  0 or 1 for deterministic algorithms. 

\begin{theorem}\label{thm:WMNFL}
  \emph{(Wolpert-Macready \cite{Wolpert97nofree}, Theorem 1)}. \\
  Assume that $\cal X$ and $\cal Y$ are finite sets. Let $f\colon \Omega \rightarrow {\cal Y}^{\cal X}$ be
  a random variable that chooses functions $f(\omega)\in {\cal Y}^{\cal X}$ with the uniform discrete 
  probability law. That means, 
  for every $h\in {\cal Y}^{\cal X}$, 
  \begin{equation*}
  P\{\omega\in\Omega:\ f(\omega)=h\} =|{\cal Y}|^{-|{\cal X}|}
  \ ,
  \end{equation*}
  where $|\cdot|$ denotes cardinality.
    
  Then, the law of $A_Y^m$ is the same for all algorithms. Precisely stated: let $A$ and $B$ be two algorithms; 
  then, for all $m\in\Naturals$ and all $y\in{\cal Y}^m$, 
\begin{equation}\label{eq:W-M}
[P\times Q] \{(\omega,\theta):\ A_Y^m(f(\omega), \theta)=y\}
  = 
[P\times Q] \{(\omega,\theta):\ B_Y^m(f(\omega), \theta)=y\}
\ . 
\end{equation}
\end{theorem}
  Since the law of $f$ is uniform, 
\begin{align*}
  & 
  [P\times Q] \{(\omega,\theta):\ A_Y^m(f(\omega), \theta)=y\}
  \\ &=
  \sum_{h\in{\cal Y}^{\cal X}} 
  Q\big\{\condprob{\theta:\ A_Y^m(h, \theta)=y }{ f(\omega)=h}\big\}
  \cdot
  P\{\omega: f(\omega)=h\}
  \\ &=
  |{\cal Y}|^{-|{\cal X}|}
  \sum_{h\in{\cal Y}^{\cal X}} 
  Q\big\{\condprob{\theta:\ A_Y^m(h, \theta)=y }{ f(\omega)=h}\big\}
  \ .
\end{align*}  
Therefore, under the hypotheses of the theorem, (\ref{eq:W-M}) can be written

  \begin{equation}\label{eq:W-M2}
   \sum_{h\in{\cal Y}^{\cal X}} 
   Q\big\{\condprob{\theta:\ A_Y^m(h, \theta)=y }{ f(\omega)=h}\big\}
  = 
   \sum_{h\in{\cal Y}^{\cal X}} 
   Q\big\{\condprob{\theta:\ B_Y^m(h, \theta)=y }{ f(\omega)=h}\big\}
  \end{equation}
for any two algorithms $A$ and $B$, and for all $y\in{\cal Y}^m$, $m\in\Naturals$.

And still another, informal, way to formulate the result of Wolpert and Macready is that
if an algorithm performs better than pure blind random search in a 
particular set of functions, then 
it must perform worse than random search \emph{in the mean} on the 
complementary set.

For simplicity, we will assume that we deal with deterministic algorithms from now on, although 
everything can be easily extended to accommodate random algorithms. 
For deterministic algorithms, equality (\ref{eq:W-M2}) can be written

  \begin{equation*}
  \Big| \big\{h \in {\cal Y}^{\cal X}:\ A_Y^m(h)=y \big\}\Big|
  = 
  \Big| \big\{h \in {\cal Y}^{\cal X}:\ B_Y^m(h)=y \big\}\Big|
  \ .
  \end{equation*}
  
\subsection{Probability concepts}
 For the sake of completeness and the reader's convenience, we summarise here, albeit in a very
 compact way, all concepts from measure and probability theory that are used in the sequel.

  Let $(\Omega, {\cal F},P)$ be a probability space and $(S, \cal S)$ any measurable space. An $S$-valued 
  random variable $f$ is a measurable mapping $f\colon \Omega\rightarrow S$.    
A \emph{random function} is simply a random variable with values in the space of all functions 
between two sets $\cal X$ and $\cal Y$, the second one equipped with some $\sigma$-field; that is, 
we take $S={\cal Y}^{\cal X}$, and the natural choice for $\cal S$ is the 
\emph{product $\sigma$-field}, i.e.
the smallest $\sigma$-field that turns every projection ${\cal Y}^{\cal X}\rightarrow {\cal Y}$ into a 
measurable mapping. Random functions
are also called \emph{stochastic processes}, especially when $\cal X$ is an interval $I$ of the real line
and $\cal Y$ is the set of real numbers $\Reals$ endowed with the Borel $\sigma$-field. We assume in the rest of the paper that ${\cal X}=[0,1]$ and ${\cal Y}=\Reals$, 
so that we are 
dealing with random functions $f(\omega)\colon [0,1]\rightarrow \Reals$. 

The \emph{law} of a stochastic process is the image measure of $P$ through the mapping $f$. That means, it 
is the probability $\mu$ on $(S,\cal S)$ such that 
\begin{equation*}
  \mu(B)=P(f^{-1}(B))
  \ ,\quad 
  \forall B\in \cal S
  \ .
\end{equation*}
  Given a stochastic process $f$, the composition  $f(t):=\delta_t\circ f$ of $f$ with the Dirac delta at 
  $t\in[0,1]$
  is automatically  a real random 
  variable $f(t)\colon \Omega\rightarrow \Reals$, with respect to the Borel $\sigma$-field on $\Reals$.
  The random vectors of the form $(f(t_1),\dots, f(t_m))$, with $t_1,\dots,t_m\in [0,1]$, are the 
  finite-dimensional  
  projections of $f$. Their laws, the \emph{finite-dimensional distributions}, determine the law of the whole process.

  A stochastic process can be represented in different ways: As a function-valued random variable, as above,
  or as a family of real-valued random variables, $\{f(t),\ t\in[0,1]\}$, or as a mapping from the product
  $\Omega\times [0,1]$ into $\Reals$, defined in the obvious way: $(\omega, t)\mapsto f(t,\omega)$.
  A process is said to be \emph{measurable} if, in the last representation, it is a measurable mapping 
  $\Omega\times [0,1]\rightarrow \Reals$
  when $\Omega\times[0,1]$ is endowed with the product $\sigma$-field ${\cal F}\times{\cal B}([0,1])$, 
  where ${\cal B}([0,1])$ denotes the Borel $\sigma$-field of $[0,1]$.
  Stochastic processes mentioned in Auger--Teytaud \cite{MR2581081} are always 
  considered measurable. This is an important hypothesis in their results, and its role will be
  made clear in the present paper. A process is said to be of \emph{second-order} if all its variables
  are square-integrable, implying that they have finite expectation and variance.
  
  With some abuse of notation, we use the same symbol $f$ to denote several different related objects:
  $f$ is a function $\Omega\rightarrow\Reals^{[0,1]}$, or $\Omega\times[0,1]\rightarrow\Reals$; for
  every $\omega\in\Omega$, $f(\omega)$ is a function $[0,1]\rightarrow\Reals$; for all $t\in[0,1]$,
  $f(t)$ is a random variable $\Omega\rightarrow\Reals$; and finally, for all $t$ and $\omega$, the
  value $f(t,\omega)$ is a real number.
  
  We will also use occasionally the customary abbreviations a.s. for \emph{almost surely}
  (i.e. true with probability 1), and a.e. for \emph{almost everywhere} (i.e. true except a set of 
  measure zero with respect to Lebesgue measure).
  
\section{Main results}\label{sec:Main}
Recall, from the notations in Section \ref{sec:Prelim}, that $A^m_Y(f(\omega))$ is the random vector 
consisting of the 
images $(f(t_1),\dots,f(t_m))(\omega)\in\Reals^m$ produced by applying $m$ iterations of algorithm $A$
on the function $f(\omega)\in\Reals^{[0,1]}$.
\begin{dfn}\label{NFLrespC}
  Let $C$ be a performance measure, measurable on $\Reals^m$, for all $m$. 
  We say that a stochastic process $f=\{f(t),\ t\in[0,1]\}$ satisfies 
  the \emph{No-Free-Lunch property 
  with respect to  $C$} if for any two algorithms $A$ and $B$, and for all $m\in\Naturals$, 
  the random variables
\begin{equation*}
\omega\mapsto C\big(A^m_Y(f(\omega))\big)\quad \text{and}\quad \omega\mapsto C\big(B^m_Y(f(\omega))\big) 
\end{equation*}  
 have the same law. 
 \end{dfn}
 Intuitively, the NFL property states that the information about $C$ that 
 we get after having sampled $m$ points 
 is the same
 no matter which algorithm we use. In particular, blind search performs as well as any other algorithm.
\begin{dfn}
  We say that a stochastic process $f=\{f(t),\ t\in[0,1]\}$ satisfies the \emph{No-Free-Lunch property} 
  if it does so with respect to all possible performance measures $C$.
\end{dfn}
\begin{remark}\label{rem:same_law}
\emph{
  It is easily seen (see e.g. Auger and Teytaud \cite{MR2581081}, Lemma 2.3), that if $f$ satisfies the No-Free-Lunch property
  then the random variables 
\begin{equation*}
\omega\mapsto A^m_Y(f(\omega)) \quad \text{and}\quad \omega\mapsto B^m_Y(f(\omega))
\end{equation*}  
have the same law. Conversely, if the above random variables have the same law, then $f$ satisfies the No-Free-Lunch 
property for all performance measures $C$.  
One may say that NFL is a extremely strong form of stationarity. A \emph{stationary process} has invariant
laws under translations: The law of $(f(t_1),\dots, f(t_n))$ and 
$(f(t_1+h),\dots, f(t_m+h))$ are the same, for every $h$ and any dimension $m$, provided all indices belong
to the set where the process is defined, the interval $[0,1]$ in our case. It is clear that the NFL property
is much stronger.
}

\end{remark}

\begin{examples}
  We can readily show two examples of random functions enjoying the No-Free-Lunch property:
\begin{enumerate}
  \item 
  Consider a set of indices $I$ of arbitrary cardinality, and a family 
  of random variables $f_i\colon\Omega\rightarrow\Reals$, $i\in I$, 
  mutually independent and identically distributed, with any non-degenerate 
  probability law. (As a specific case, consider for instance 
  Bernoulli or Gaussian variables, and $I=\Reals^+$ as the set of indices).
  
  Such a stochastic process exists by the classical Kolmogorov Extension Theorem
  (see e.g. \cite{MR0435320}). All $n$-dimensional joint distributions
  are the direct product measure of the individual laws, and are therefore
  the same. The NFL property is then trivially satisfied.
  
  \item
  Consider any random variable $X$ and define a constant process $f(t)\equiv X$, for all 
  $t\in[0,1]$. The NFL property is also immediate to check.
\end{enumerate}
  These examples show that there exist NFL situations also in the continuum case, and 
  that the cardinality of $I$ alone cannot be the responsible of the lack of 
  No-Free-Lunch. 

\end{examples}

 It is certainly true that a continuous-time stochastic process with all variables mutually independent can hardly 
 be of any interest in modelling a real phenomenon (in sharp contrast with the discrete case). Notice for example, that
 in the common case of Gaussian variables, almost all sample paths (i.e., all functions in the set we are considering, 
 with probability 1) are unbounded from below and from above, which makes 
 pointless to search for or to approximate 
 the minimum value. As another example, if 
 the variables have the uniform law in an interval $[\alpha, \beta]$,
 then almost all sample paths are bounded, with infimum equal to $\alpha$ and supremum 
 equal to $\beta$,  although the probability that these values are attained is zero. 
 A very different problem is the case when $\alpha$ or $\beta$ are unknown and we try to estimate them
 by sampling $m$
 points from independent variables distributed uniformly in $[\alpha,\beta]$;
 this is indeed a statistical problem of a real practical interest.

 Such trivial NFL situations do not appear if we impose on $f$ the condition of being a measurable process. 
 This is the main result of Auger and Teytaud \cite{MR2581081}. 
 We reformulate it as the problem of finding
 a necessary condition for having NFL in a measurable process, and show that in this case we 
 are dealing essentially with a constant process. This possibility does not appear in 
 \cite{MR2581081},
 because of the hypothesis of existence of a so-called ``proper median'', that the 
 authors introduce in the definition of NFL, and that it looks somewhat artificial.
 We will not use this concept.
 We also point out that the argument in \cite{MR2581081} 
 is in our opinion not complete, since at some point in the proof of their Theorem 4.1
 there is a confusion between the underlying randomness of the process and the eventual
 randomness of the algorithm applied. 
 
 We start by showing that measurability and independence together collapses the process to 
 an essentially discrete-time stochastic process. 
 \begin{theorem}\label{thm:MesInd}
 Let $f=\{f(t),\ t\in[0,1]\}$ be a measurable stochastic process with 
 mutually independent
 random variables $f(t)$.
 Then, for almost all $t$ with respect to Lebesgue measure, the
 random variable $f(t)$ is constant with probability 1.
 
%  The same holds for a second-order process with mutually uncorrelated
%  random variables. \Comment{((For the moment, we have not proved it: In the general case, when
%  truncating, the truncated variables are not uncorrelated; the correlation tends to zero
%  when $m\to\infty$; try to use this? If not, take out the uncorrelated case))}
\end{theorem}

\begin{proof}
We treat first the particular case in which the process is bounded and centred. Then the
result will be easily extended to the general case. 

\emph{First case:} Assume $\E[f(t)]=0$ and that there is a constant $m\in\Naturals$ such that 
for all $t$, $|f(t)|<m$. 

If $f\colon \Omega\times [0,1]\rightarrow \Reals$ is a measurable mapping, then the partial mappings 
$f(\omega)\colon [0,1]\rightarrow \Reals$ are also measurable, for $\omega\in\Omega$ almost surely. 
Since, moreover, all the sample paths are bounded, it makes sense to consider their Lebesgue integrals  
\begin{equation*}
g(t)=\int_0^t f(s)\,ds\ , \quad t\in[0,1]
\ .
\end{equation*}
From the properties of the integral, $g(t)$ is a process with continuous paths almost surely. 
We may leave it undefined for the exceptional set of probability zero, because this is not relevant.

The random variables $g(t)$ are also clearly bounded (e.g. by $m$ itself), and therefore the 
second moment $\E[g(t)^2]$ is finite. But we see that it is in fact equal to zero:
\begin{align*}
  \E[g(t)^2]
  &=
  \E\Big[\int_0^t f(s)\,ds\cdot\int_0^t f(r)\,dr\Big]
  =
  \E\Big[\int_0^t\int_0^t f(s)f(r)\,dsdr\Big]
  =
  \int_0^t\int_0^t \E[f(s)f(r)]\,dsdr
  =
  \\
  &=  
  \int_0^t\int_0^t\E[f(s)^2]\cdot\ind_{\{s=r\}}\,dsdr+\int_0^t\int_0^t \E[f(s)]\E[f(r)]\cdot\ind_{\{s\neq r\}}\,dsdr
  \ ,
\end{align*}
  where the interchange of integral and expectation is justified by the boundedness of all functions involved, 
  which allows to apply the Fubini theorem, and we have used the hypothesis 
  $\E[f(s)\cdot f(r)]=\E[f(s)]\cdot 
  \E[f(r)]$. 
  Now, the first integral is equal to zero because we are integrating over the line $\{s=r\}$, which has zero Lebesgue measure,
  and the second one is also zero because the variables are centred.
  
  The equality $\E[g(t)^2]=0$ implies that for all $\omega\in\Omega$ except maybe in a subset $N_t\subset\Omega$ 
of probability zero, one has $g(t,\omega)=0$. In particular, this is true for 
all $t\in\Rationals$ and, since $\Rationals$ is a countable set, we have that $N=\unio_{t\in\Rationals} N_t$  
has probability zero. 
%((Therefore, for all $\omega\in\Omega-N$, and for all $t\in\Rationals$, $g(t,\omega)=0$.))
By the continuity of the paths, we obtain that for $\omega\in\Omega-N$, $g(t,\omega)=0$ for all $t\in [0,1]$. 
The integral being zero for all $t$, the integrand is also zero except maybe on a set of 
Lebesgue measure equal to zero. 
We conclude that a.s., and for almost all $t\in[0,1]$, 
$f(t,\omega)=0$. 

\emph{General case:} \\ 
Let $m\in\Naturals$, and define:
$$\bar{f}_{m}(t):=f(t)\cdot\textbf{1}_{\{|f(t)|<m\}}$$
and
$$\hat{f}_{m}(t)=\bar{f}_{m}(t)-\E[\bar{f}_{m}(t)]\ .$$

The random variables $\hat f(t)$ are also mutually independent. Applying the particular case above 
we have that for all $m$, the law of
$\hat{f}_{m}(t)$ is a Dirac delta at zero, for $t\in[0,1]$ a.e., and therefore 
$\bar{f}_{m}(t)$ is constant $\omega$-a.s, $t$-a.e.

Now, let $N_{t,m}\subset \Omega$ be the set where $\bar f_{m}(t)\neq f(t)$. 
%((The sequence $\{N_m(t)\}_{m\in\Naturals}$
%decreases to a set $N_\infty(t)$ of probability zero, since $\lim_{m\to\infty} P\{|f(t)|\le m\}=1$. Therefore,))
The random variable $f(t)$ is almost surely equal to a constant $a(t,m)$ on the set $\Omega-N_{t,m}$, 
which tends, as $m\to\infty$, to
a set $\Omega-N_\infty(t)$, whose probability is 1, given that $\lim_{m\to\infty} P\{|f(t)|\le m\}=1$. 
We obtain that the constant $a(t,m)$ cannot depend on $m$, and conclude that
for almost all $t\in[0,1]$, with respect to Lebesgue measure, the random variables
$f(t)$ are degenerated. 
\end{proof}
\par
\bigskip
In other words, Theorem \ref{thm:MesInd} states that, under the hypotheses of measurability and 
independence% 
%\Comment{((or uncorrelation?))}
, randomness can only appear 
on a time set of zero Lebesgue measure. Hence, the process
can be considered, in a measure-theoretic sense, as a discrete time stochastic process.

%\Comment{((Review par if we finally drop second-order hypothesis in Theorem \ref{thm:NFLisConst}))} 
For the remaining of the section, we assume that we deal with second-order processes. 
We will show, in Theorem \ref{thm:NFLisConst} below, that a measurable, second-order process satisfying the NFL property
is trivial: All their random variables are almost surely equal.

First, we state some preliminary results in the form of lemmas:
\begin{lemma}
\label{th1ITA}
Let $(\Omega,\mathcal{F},P)$ be a probability space and 
$f : \Omega \times [0,1] \rightarrow \mathbb{R}$ a measurable second-order stochastic process.
Then the NFL property can be satisfied only if the random variables $f(t)$ are identically 
distributed and the covariance $\Cov(f(t),f(s))$ 
is constant for all $t$, $s \in [0,1]$, $t \neq s$. 
\end{lemma}

\begin{proof}  
If the NFL property is satisfied, then by Remark \ref{rem:same_law} the vectors $A_Y^m(f)$ and 
$B_Y^m(f)$ 
are identically distributed for any $m\in\Naturals$ and any pair of algorithms $A$
and $B$.

Take $m=1$. Given two values $t$ and $s$ in $[0,1]$, let $A$ be a deterministic algorithm that chooses 
$t$ as initial point and $B$ another
deterministic algorithm that chooses $s$ as initial point.
Then $f(t)$ and $f(s)$ are identically distributed.

Take now $m = 2$. 
Given two couples of different points  $(t_1,t_2)$ and $(s_1,s_2)$, let $A$ be a deterministic algorithm 
that choses $(t_1,t_2)$ as the first two points 
and $B$ be a deterministic algorithm that chooses $(s_1,s_2)$.
Then the random vectors $(f(t_1),f(t_2))$ and $(f(s_1),f(s_2))$ are identically distributed. 
This fact implies in particular that 
$$\Cov(f(t_1),f(t_2)) = \Cov(f(s_1),f(s_2)) \ .$$
\end{proof}

\begin{lemma}\label{lem:rho_bound}
Let $X_1, \dots, X_n$ ($n\ge 2$) be real square-integrable random variables with a common covariance
$\rho = \Cov[X_i,X_j]$ when $i\neq j$, and define $V:=\max_i \Var[X_i]$, the maximum of their variances.
Then $$\rho \ge - \frac{V}{n-1}$$  
\end{lemma}

\begin{proof}
  We have 
\begin{equation*}
  0 
  \le 
  \Var\Big[\sum_{i=1}^n X_i\Big]  
  =
   \sum_{i=1}^n \Var [X_i] 
  +
  \sum_{\substack{i,j=1 \\ i\neq j}}^n \Cov[X_i, X_j]
  \le 
  Vn + n(n-1)\rho
\end{equation*}
  and the result follows at once. 
\end{proof}

The proof of the next result is immediate:
\begin{lemma}\label{red_process}
Let $f\colon \Omega \times[0,1] \rightarrow \mathbb{R}$ be
a second-order stochastic process such that the variables  
$f(t)$ are identically distributed, with finite common mean $\mu$ and positive common variance $\sigma^2$. 
Then $f$ satisfies the NFL property if and only if the same holds for $(f-\mu)/\sigma$.
\end{lemma}

By Lemmas \ref{th1ITA}, \ref{lem:rho_bound} and \ref{red_process}, 
we can restrict the search of a second-order measurable stochastic process 
satisfying the NFL property to the case when 
the variables $f(t)$ are identically distributed, with zero mean, unit variance and such that 
for some $\rho\ge 0$, $\Cov[f(t),f(s)] = \rho$ for any pair $t \neq s$.

In the next theorem we use Fourier analysis, following quite closely some arguments that can be 
found in Crum \cite{MR0083534}, to prove our main result:
\begin{theorem}\label{thm:NFLisConst}
  Let $\{f(t,\omega):\ t\in[0,1],\  \omega\in\Omega\}$ be a measurable, second-order
  stochastic process, with $\E[f(t)]=0$ and $\E[f(t)^2]=1$ for all $t\in[0,1]$, 
  satisfying the NFL property, and defined in some
  probability space $(\Omega, {\cal F}, P)$.

  Then, the process is constant, in the sense that there exists a random variable $X\colon\Omega\rightarrow\Reals$
  such that $P\{\omega\in\Omega:\ f(t,\omega)=X(\omega)\}=1$, $\forall t\in[0,1]$.

\end{theorem}
\begin{proof}
We are going first to extend the process from $[0,1]$ to the whole real line, in order to apply
Fourier transform techniques comfortably. Define:
\begin{equation*}
f(k+t):=
\begin{cases}
  f(t),\ \text{if $t\in (0,1]$, $k=1,2,\dots$} 
  \\
  f(t),\ \text{if $t\in [0,1)$, $k=-1,-2,\dots$} 
\end{cases}
\end{equation*}

From the previous lemmas, we know that the covariance function $K(t):=\E[f(t+s)f(s)]$ of
the extended process is equal to some $\rho$ ($0\le\rho\le 1$) for 
$t$ in an interval $(-\varepsilon,\varepsilon)$, and all $s\in\Reals$, except for $t=0$, in which is equal to 1. 
Our purpose is to see that in fact $\rho$ must be equal to 1, from which the conclusion will be easily 
drawn.

  It can be readily seen that the extended process is also measurable, using that a set 
  $A\in{\cal B}(I)\otimes{\cal F}$, where $I$ is any interval in $\Reals$, 
  is also ${\cal B}(\Reals)\otimes{\cal F}$-measurable when considered as
  a subset of $\Reals\times\Omega$.
  %\Comment{((easy enough?; more details?))} 
  This implies that, for almost all $\omega$, 
  $t\mapsto f(t,\omega)$ is 
  a Borel measurable function, that it makes sense to consider the integrals
\begin{equation*}
  \int_\Reals e^{-2s^2} f(s)^2\, ds
  \ ,
\end{equation*}
  and that they are measurable functions $\Omega\rightarrow\Reals$. 
  Taking expectation and applying Fubini's theorem, 
\begin{equation*}
  \E\Big[\int_\Reals e^{-2s^2} f(s)^2\, ds\Big]
  =
  \int_\Reals e^{-2s^2} \E[f(s)^2]\, ds
  =
  \int_\Reals e^{-2s^2} \, ds 
  <
  \infty
  \ .
\end{equation*}
  This means that, $\omega$-a.s., $s\mapsto e^{-s^2}f(s)$ belongs to $L^2(\Reals)$.
  It also belongs to $L^1(\Reals)$, $\omega$-a.s:
\begin{equation*}
  \E\Big[\int_\Reals e^{-s^2} |f(s)|\, ds\Big]
  =
  \int_\Reals e^{-s^2} \E[|f(s)|]\, ds
  \le
  \int_\Reals e^{-s^2} \E[f(s)^2]^{1/2}\, ds 
  =
  \int_\Reals e^{-s^2} \, ds 
  <
  \infty
  \ .
\end{equation*}
  For such $\omega$, consider the function 
\begin{equation*}
  g(s):= e^{-s^2}f(s)-e^{-(s+t)^2} f(s+t)
  \ .
\end{equation*}

Since $g\in L^1(\Reals)$, we may take its Fourier transform
\begin{equation*}
  \hat g(\xi)
  =
  \int_\Reals g(s)\cdot e^{-2\pi is\xi}\, ds
\end{equation*}
  which can be written as
\begin{equation*}
  \int_\Reals e^{-s^2}f(s)\cdot e^{-2\pi is\xi}\, ds
  -
  \int_\Reals e^{-s^2}f(s)\cdot e^{-2\pi i(s-t)\xi}\, ds
  \ ,
\end{equation*}
or 
\begin{equation*}
  (1-e^{-2\pi i t\xi})\cdot \hat F(\xi)
  \ ,
\end{equation*}
  where $\hat F$ is the Fourier transform of $s\mapsto e^{-s^2}f(s)$.

  Since $g\in L^2(\Reals)$ also, by Plancherel's Theorem,
\begin{equation*}
  G(t,\omega)
  :=
  \int_\Reals g(s)^2\,ds
  =
  \int_\Reals |\hat g(\xi)|^2\,d\xi
  =
  \int_\Reals |1-e^{-2\pi i t\xi}|^2 \cdot |\hat F(\xi)|^2\,d\xi
  \ .
\end{equation*}
  The integrand tends to zero as $t\to 0$, and it is dominated by
  $4\cdot |\hat F(\xi)|^2\in L^1(\Reals)$. Therefore, 
  $\lim_{t\to 0} G(t,\omega)=0$, a.s.
  
Moreover,
\begin{align*}
  G(t,\omega)
  &=
  \int_\Reals \big|e^{-s^2}f(s)-e^{-(s+t)^2} f(s+t)\big|^2\, ds
  \\
  &\le
  2 \Big[\int_\Reals e^{-2s^2}f(s)^2\, ds + \int_\Reals e^{-2(s+t)^2}f(s+t)^2\, ds\Big]
  = 
  4\int_\Reals e^{-2s^2}f(s)^2\, ds 
  \ ,
\end{align*}
  that belongs to $L^1(\Omega)$, as we have seen before. By the Dominated 
  Convergence Theorem again,
\begin{equation*}
  \lim_{t\to 0} \E[G(t,\omega)] = 0 
  \ .
\end{equation*}
  On the other hand, 
\begin{align*}
  G(t,\omega)
  &=
  \int_\Reals e^{-2s^2}f(s)^2 \, ds
  +
  \int_\Reals e^{-2(s+t)^2}f(s+t)^2 \, ds
  -
  2\int_\Reals e^{-s^2-(s+t)^2}f(s)f(s+t) \, ds
  \\
  &=
  2\int_\Reals e^{-2s^2}f(s)^2 \, ds
  -
  2\int_\Reals e^{-s^2-(s+t)^2}f(s)f(s+t) \, ds  
  \ .
\end{align*}
  The expectation of the first term is equal to 
\begin{equation*}
  2\int_\Reals e^{-2s^2}\, ds = \sqrt{2\pi}
  \ .
\end{equation*}
  For the second, it yields
\begin{equation*}
  2\int_\Reals e^{-s^2-(s+t)^2} K(t)\, ds = \sqrt{2\pi}e^{-t^2/2} K(t)
\end{equation*}
  (the interchange of integral and expectation is justified here by
  checking first the integrability).
  
  We get 
\begin{equation*}
  \E[G(t,\omega)]=\sqrt{2\pi} \big(1-e^{-t^2/2}K(t)\big)
  \ .
\end{equation*}
  Hence
\begin{equation*}
  0=\lim_{t\to 0} \E[G(t,\omega)]
  =
  \lim_{t\to 0} \sqrt{2\pi}\big(1-e^{-t^2/2}K(t)\big)
  \ ,
\end{equation*}
 which implies that $\lim_{t\to 0} K(t) = 1$, and we conclude that
 $\rho = 1$, as we wanted to see.
 
 Finally, since $\Cov[f(t),f(s)]=1$, we have $f(t)=\alpha f(s)+\beta$, for some $\alpha,\beta$.
  But the variables are centred, and this implies $\beta=0$, whereas the unit variances
  yield $|\alpha|=1$. The negative value of $\alpha$ is impossible because the covariance is nonnegative.
  Hence $f(t)=f(s)$, almost surely, for all $t$ and $s$, and the proof is complete. 
\end{proof}
\begin{corollary}
  In view of Lemma \ref{red_process}, the conclusion of Theorem \ref{thm:NFLisConst} is true without the 
  hypotheses of null expectation and unit variance. 
\end{corollary}

  %\Comment{((Can we drop the second-order hipotheses of the theorem  and corollary??))}

  Notice that the conclusion of the previous theorem and corollary does not mean that almost all sample
  paths are constant, because the null set $N_t$ where the equality $f(t)=X$ fails depends on $t$. 
  However, in an optimisation setting it is natural to specify a regularity assumption on the 
  functions, 
  besides the probabilistic model. For example, the continuity of the paths 
  (or simply the right or left continuity)
  automatically yields that the union $\cup_{t\in[0,1]}N_t$ has probability zero, and in that 
  case one may say that the process is constant in the sense that, except on a set $N\subset\Omega$ of probability
  zero, all paths $f(t)$ are constant.

Summarizing the present section:
\begin{itemize}
  \item 
Without measurability assumptions, we showed two examples of 
NFL property in the continuum: The case in which all variables are independent, identically
distributed, and the case where the process is constant: $f(t)= X$ a.s, 
$\forall t$, for some random variable $X$. Both are unimportant from 
the optimisation practitioner's point
of view, and both are ruled out in the mentioned paper by Auger and Teytaud, 
by imposing the measurability and the ``proper median'' hypotheses, respectively.

  \item 
We have shown that measurability and independence together lead to a ``discrete time'' process,
and that measurability and NFL (for second-order processes)
%\Comment{((drop))?}) 
imply that the process is constant. With the three conditions 
together, or simply measurability, independence and stationarity, one gets easily that  
each variable of the process must be almost surely equal to some constant $k$, the
same for all of them: $P\{f(t)=k\}=1$, $\forall t$.
\end{itemize}

\section{Conclusion and open questions}\label{sec:Conc}
It is frequently argued that in realistic scenarios the hypotheses of the NFL theorems are always violated,
already in the genuinely discrete cases. We have shown that in the continuum the necessary conditions for NFL are even
stronger and far too restrictive to be found in practice. 

This means that in every practical situation 
there must be some information on the objective function that permits, in principle, to
choose algorithms that perform better than pure blind search.
We believe that the usefulness of the (no)-NFL statements is precisely on the theoretical side, to highlight that
any proposal of a search algorithm, supported by a benchmark of functions in which it behaves well, should 
be accompanied by a study of the benchmark common features that help that algorithm beat 
the others.
In other words, as has been emphatically pointed out in a recent expository article
\cite{Serafino}: <<It is clear now that for the practitioner
the correct question is not \emph{which algorithm I have
to use} but first of all \emph{what is the geometry of
the objective function}>>.

Before paper \cite{MR2581081}, there was, to our knowledge, no special interest in investigating the 
existence
of NFL theorems in the continuum. The typical argument was that
only the discrete case is important in practice, since the 
computations are always made in finite precision. This is no longer true, since arbitrary 
precision computations are commonplace today, meaning that the underlying computational
model is at least the rational number system $\Rationals$.

There are still two questions that deserve further study concerning NFL theorems, both in the discrete
and in the continuum settings:

\begin{itemize}
  \item 
The first one is the consideration of noisy functions, that means, black-box 
functions that may answer differently when asked twice for the value at the same point $t$. This is 
not uncommon in practice, since the computation of the objective value at a feasible point may involve
itself some randomness or the heuristic solution of another optimisation problem. In that case, the
algorithms to consider should be allowed to sample more than once the same point.

  \item 
More importantly, the second question refers to the concept of No-Free-Lunch itself. As we have seen, 
the NFL property is so strong that 
constant processes are the only measurable processes that qualify.
%the only measurable processes that satisfy it are the constant ones. 
But if we are just concerned with minimizing a function, the relevant performance measure 
is $C(y_1,\dots,y_n)=\min\{y_1,\dots,y_n\}$, or perhaps some related function. 
Recalling Definition \ref{NFLrespC} applied to this measure, it is easy to see that NFL with respect to
$C$ is not sufficient to conclude that the laws of $\omega\mapsto A^m_Y(f(\omega))$ and 
$\omega\mapsto B^m_Y(f(\omega))$ have to be the same,
and then our Theorem \ref{thm:NFLisConst} need not be true. We believe that this point deserves 
further investigation.
\end{itemize}

\section{Acknowledgements}
  This work has been supported by grants numbers MTM2011-29064-C03-01  
  from the Ministry of Economy and Competitiveness of Spain; 
  UNAB10-4E-378, co-funded by the European Regional Development Fund (ERDF); 
  and 60A01-8451 from the University of Padova.
  %; and \Comment{((Marco and Alessandro grants))}
\bibliographystyle{plain}
\bibliography{Alabert-Berti-Caballero-Ferrante2.1}   % name your BibTeX data base     

\end{document}